\documentclass[12pt]{amsart}
\usepackage{a4}
\usepackage{amsmath,amssymb,amscd,amsthm}
\usepackage{amsfonts}
\usepackage[dvipdfm]{color}

\newtheorem{thm}{Theorem}[section]
\newtheorem{prop}[thm]{Proposition}
\newtheorem{lem}[thm]{Lemma}

\theoremstyle{definition}
\newtheorem{defi}[thm]{Definition}
\newtheorem{rem}[thm]{Remark}

\newcommand{\R}{\mathbb{R}}
\newcommand{\g}{\mathfrak{g}}
\newcommand{\M}{\widetilde{\mathfrak{M}}}
\newcommand{\PM}{\mathfrak{PM}}
\newcommand{\RH}{\mathbb{R} \mathrm{H}}
\newcommand{\1}{\g_{\RH^2} \oplus \mathbb{R}^{n-2}}
\newcommand{\2}{\g_{\RH^{n-1}} \oplus \mathbb{R}}
\newcommand{\A}{\mathbb{R}^{\times}\mathrm{Aut}(\g)}
\newcommand{\Aut}{\mathrm{Aut}(\g)}
\newcommand{\D}{\mathbb{R} \oplus {\rm Der}(\mathfrak{g})}
\newcommand{\Der}{\mathrm{Der}(\g)}
\newcommand{\naiseki}{\langle , \rangle}
\newcommand{\GL}{\mathrm{GL}_n(\R)} 
\newcommand{\OO}{\mathrm{O}(n)} 
\numberwithin{equation}{section}

\author{Takahiro Hashinaga}
\address[T.~Hashinaga]{Department of Mathematics, Hiroshima University, 
Higashi-Hiroshima 739-8526, Japan}
\email{hashinaga@hiroshima-u.ac.jp}

\author{Hiroshi Tamaru}
\address[H.~Tamaru]{Department of Mathematics, Hiroshima University, Higashi-Hiroshima 739-8526, Japan}
\email{tamaru@math.sci.hiroshima-u.ac.jp}

\author{Kazuhiro Terada}
\address[K.~Terada]{Department of Mathematics, Hiroshima University, Higashi-Hiroshima 739-8526, Japan}
\curraddr{SystemSoft Corp., Minato-ku, Tokyo 105-0013, Japan.}

\thanks{
The first author was supported in part by Grant-in-Aid for JSPS Fellows (11J05284). 
The second author was supported in part by KAKENHI (20740040, 24654012).
} 

\title[Milnor-type theorems for left-invariant metrics]
{Milnor-type theorems for left-invariant Riemannian metrics on Lie groups}

\begin{document}

\begin{abstract}
For all left-invariant Riemannian metrics on three-dimensional unimodular Lie groups, 
there exist particular left-invariant orthonormal frames, so-called Milnor frames. 
In this paper, for any left-invariant Riemannian metrics on any Lie groups, 
we give a procedure to obtain an analogous of Milnor frames, 
in the sense that the bracket relations among them can be written with relatively smaller number of parameters. 
Our procedure is based on the moduli space of left-invariant Riemannian metrics. 
Some explicit examples of such frames and applications will also be given. 
\end{abstract}

\maketitle

\section{Introduction\label{sec1}}

For every left-invariant Riemannian metrics 
on three-dimensional unimodular Lie groups, 
Milnor (\cite{Mil}) constructed certain orthonormal basis 
of the corresponding metric Lie algebras. 
Such bases are nowadays called the \textit{Milnor frames}, 
and have played crucial roles in many branches of geometry. 
For example, 
the curvatures of left-invariant Riemannian metrics on such Lie groups 
can be calculated explicitly in terms of the Milnor frames. 
As a consequence, one can determine all possible signatures for the Ricci curvatures 
in this case (\cite{Mil}). 
Furthermore, in terms of Milnor frames, 
one can also study the Ricci flow and Ricci solitons 
(\cite[Chapter~1]{CK}, see \cite{LW2013} for more information and references). 
For left-invariant Einstein and Ricci soliton 
metrics on Lie groups, 
we refer to, 
for example, 
\cite{Heber, Lau09, Lau11, L.W, T05, T08, T11}. 

Since Milnor frames are quite powerful tools, 
it is desired to construct a generalization of Milnor frames 
for other Lie groups, 
which might be useful for studies in many areas. 
Note that Milnor's original arguments strongly depend on dimension three, 
but some generalizations have been known: 
for example, 
Chebarikov (\cite{Ch}), and
Ha and Lee (\cite{HL}) 
studied three-dimensional non-unimodular Lie groups, 
and Kremlev and Nikonorov (\cite{KN1, KN2}) studied four-dimensional cases. 
There are some related studies for nilpotent Lie algebras, 
in the framework of ``Ricci-diagonal basis'', 
which we refer to Payne (\cite{P}), Lauret and Will (\cite{LW2013}), and references therein. 

In this paper, 
we consider an arbitrary Lie group $G$ with Lie algebra $\g$, 
and give a procedure to construct an analogous of Milnor frames. 
More precisely, we give a procedure to obtain the following kind of theorem: 

\begin{itemize} 
\item[] 
For any inner product $\naiseki$ on $\g$, 
there exists an orthonormal basis 
$\{ x_1, \ldots , x_n \}$ with respect to $\naiseki$ (up to scaling) 
such that the bracket relations among them 
can be written with relatively smaller number of parameters. 
\end{itemize} 
In this paper, this kind of theorem is called a \textit{Milnor-type theorem} for $\g$. 
Our procedure will be described in Section~\ref{sec.2}. 
In fact, it 
is based on the moduli space of left-invariant metrics, 
and is a generalization of the methods in \cite{Ch, HL, KN1, KN2}. 

By applying our procedure, in Section~\ref{sec.3}, 
we obtain Milnor-type theorems for two $n$-dimensional Lie algebras, 
namely $\1$ and $\2$. 
Furthermore, 
we study the curvatures of these Lie algebras in Section~\ref{sec.4}. 
In particular, 
we determine all possible signatures for the Ricci curvatures of these Lie algebras, 
and also study whether they admit left-invariant Ricci solitons. 
We note that these two Lie algebras are just toy models, 
but already show that Milnor-type theorems are useful 
for studying some higher-dimensional Lie algebras. 

Finally in this section, 
we emphasize that our procedure can be applied, 
at least theoretically, to an arbitrary Lie algebras. 
For example, by using our procedure, 
Milnor-type theorems for all three-dimensional solvable Lie algebras 
are obtained in the forthcoming paper \cite{HT}. 

The authors would like to thank Yoshio Agaoka and Kazuhiro Shibuya 
for useful comments and suggestions. 
They are also grateful to Akira Kubo for some discussions.  
Finally, the authors would like to thank the referee for pointing out the reference \cite{Ch}.

\section{A general procedure}
\label{sec.2}

In this section, we describe a procedure to obtain 
a Milnor-type theorem for an arbitrary Lie algebra $\g$. 
Our main theorem states that, 
a Milnor-type theorem for $\g$ can be obtained from the moduli space $\PM$ of left-invariant Riemannian metrics. 
Note that the space $\PM$ has been introduced and studied in \cite{KTT}. 

\subsection{Preliminaries}

\label{subsection:pm}

First of all, we recall the moduli space of left-invariant Riemannian metrics. 
We refer to \cite{KTT} for details. 

Let $G$ be a Lie group, and $\g$ be the Lie algebra of $G$. 
We consider the set of all left-invariant Riemannian metrics on $G$, 
which can naturally be identified with 
\begin{align}
\M := \{ \naiseki \mid \mbox{an inner product on $\g$} \} .  
\end{align}
Let $n := \dim \g$, and identify $\g \cong \R^n$ as vector spaces. 
For $\naiseki \in \M$ and $g \in \GL$, we define
\begin{align}
g.\langle \cdot, \cdot \rangle := \langle g^{-1} (\cdot), g^{-1} (\cdot) \rangle . 
\end{align}
This induces a transitive action of $\GL$ on $\M$. 
We thus have an identification 
\begin{align}
\M \cong \GL / \OO . 
\end{align} 
Note that $\M$ is a noncompact Riemannian symmetric space, 
by equipping with a certain $\GL$-invariant metric 
(for example, see \cite[Subsection~4.1]{KTT}). 

In order to define the moduli space $\PM$, 
let us consider the automorphism group and the scalar group: 
\begin{align}
\Aut & := \{ \varphi \in \GL \mid \varphi[\cdot, \cdot] = [\varphi (\cdot), \varphi (\cdot)] \} , \\ 
\mathbb{R}^{\times} & := 
\{ c \cdot \mathrm{id} : \g \rightarrow \g \mid c \in \R \setminus \{ 0 \} \} . 
\end{align}
The group $\A$ naturally acts on $\M$. 
Note that the action of $\R^\times$ gives rise to a scaling, 
and the action of $\Aut$ induces an isometry 
of the corresponding left-invariant metrics. 

\begin{defi}\label{DefPM}
The orbit space of the action of $\A$ on $\M$ is called the 
\textit{moduli space of left-invariant Riemannian metrics}, 
and denoted by 
\begin{align}
\PM := \A \backslash \M . 
\end{align}
\end{defi}

Note that the action of $\A$ on $\M$ is isometric 
with respect to $\GL$-invariant metrics. 
Hence, in order to study $\PM$ 
(for example, possible topological type), 
one can use general theories of isometric actions on symmetric spaces. 

\subsection{A set of representatives} 

Our main theorem states that 
an expression of $\PM$ derives a Milnor-type theorem. 
Here, by an expression of $\PM$, 
we mean a set of representatives is given. 
In this subsection, we formulate a set of representatives. 

Let $\naiseki_0$ be the canonical inner product on $\g \cong \R^n$. 
For simplicity of the notation, 
the orbit of $\A$ through $\naiseki$ is denoted by 
\begin{align}
[\naiseki] := \A . \naiseki 
:= \{ \varphi.\naiseki \mid \varphi \in \A \} . 
\end{align}

\begin{defi}
A subset $U \subset \GL$ is called a 
\textit{set of representatives} of $\PM$ if it satisfies 
\begin{align}
\PM = \{ [h. \naiseki_0] \mid h \in U \} . 
\end{align}
\end{defi}

By a set of representatives, 
we do not mean that it is a complete set of representatives. 
But of course, 
it is expected that $U$ is chosen to be as small as possible. 

We here have a criteria for $U$ to be a set of representatives. 
Let $[[g]]$ denote the double coset of $g \in \GL$, defined by 
\begin{align} 
[[ g ]] := \A \, g \, \OO := \{ \varphi g k \mid \varphi \in \A , \, k \in \OO \} . 
\end{align}

\begin{lem}
\label{lem:representatives}
Let $U \subset \GL$. 
Then the following are mutually equivalent$:$ 
\begin{enumerate} 
\item 
$U$ is a set of representatives of $\PM$. 
\item
For every $g \in \GL$, there exists $h \in U$ such that 
$[h.\naiseki_0] = [g.\naiseki_0]$. 
\item
For every $g \in \GL$, there exists $h \in U$ such that $h \in [[g]]$. 
\end{enumerate} 
\end{lem} 

\begin{proof}
It is easy to see that (1) and (2) are equivalent. 
In order to prove that (2) and (3) are equivalent, 
we have only to show that 
$[h.\naiseki_0] = [g.\naiseki_0]$ and $h \in [[g]]$ are equivalent. 

Assume $[h.\naiseki_0] = [g.\naiseki_0]$. 
Then there exists $\varphi \in \A$ such that 
\begin{align}
h.\naiseki_0 = \varphi.(g.\naiseki_0) = (\varphi g).\naiseki_0 . 
\end{align}
This yields that $(\varphi g)^{-1} h \in \OO$. 
One thus has 
\begin{align}
h = \varphi g ((\varphi g)^{-1} h) \in \A \, g \, \OO = [[g]] . 
\end{align}

Conversely, assume $h \in [[g]]$. 
Hence there exist $\varphi \in \A$ and $k \in \OO$ such that $h = \varphi g k$. 
This yields that 
\begin{align}
h.\naiseki_0 = (\varphi g k).\naiseki_0 = \varphi.(g.\naiseki_0) . 
\end{align}
This shows $[g.\naiseki_0] = [h.\naiseki_0]$. 
\end{proof} 

\subsection{Main theorem}

We are now in the position to prove our main theorem of this paper. 
Namely, a set of representatives of $\PM$ derives a Milnor-type theorem. 
Recall that $\naiseki_0$ is the canonical inner product on $\g \cong \R^n$. 
Denote by $\{ e_1, \ldots, e_n \}$ the canonical orthonormal basis. 

\begin{thm}\label{MTT}
Let $U$ be a set of representatives of $\PM$. 
Then, for every inner product $\naiseki$ on $\g$, 
there exist $h \in U$, $\varphi \in \Aut$, and 
$k>0$ 
such that 
$\{ \varphi h e_1 , \ldots , \varphi h e_n \}$ 
is an orthonormal basis of $\g$ 
with respect to $k \naiseki$. 
\end{thm}

\begin{proof}
Take any inner product $\naiseki$ on $\g$. 
Since $U$ is a set of representatives of $\PM$, 
there exists $h \in U$ such that 
\begin{align}
[\naiseki] = [h.\naiseki_0] . 
\end{align} 
Recall that $[ \cdot ]$ denotes the orbit of $\A$. 
Hence, there exist $c \in \R^\times$ and $\varphi \in \Aut$ 
such that
\begin{align}
\naiseki = (c \varphi).(h.\naiseki_0) = (c \varphi h).\naiseki_0 . 
\end{align} 
Take any $i, j$. 
Then, it follows from the definition of the action that 
\begin{align}
\begin{split}
\langle 
\varphi h e_i , \varphi h e_j 
\rangle 
& = (c \varphi h). \langle \varphi h e_i , \varphi h e_j \rangle_0 \\ 
& = \langle (c \varphi h)^{-1} \varphi h e_i , 
(c \varphi h)^{-1} \varphi h e_j \rangle_0 \\ 
& = c^{-2} \delta_{i j} .  
\end{split}
\end{align}
Hence, by putting $k := c^2 > 0$, we complete the proof. 
\end{proof}

When we apply this theorem for a given Lie algebra $\g$, 
we put $x_i := \varphi h e_i$ and study the bracket relations among them. 
Note that $\varphi$ preserves the bracket product. 
Hence, if $U$ contains only $l$ parameters, 
then so do the bracket relations among $\{ x_1 , \ldots , x_n \}$. 
This is a procedure to obtain Milnor-type theorems. 
Examples of Milnor-type theorems will be given in the following section.

\section{Examples of Milnor-type theorems}
\label{sec.3}

In this section, 
we apply our main theorem to two particular Lie algebras, namely $\1$ and $\2$, 
and derive Milnor-type theorems for them. 

\subsection{Preliminaries}

In this subsection, we recall known facts on the Lie algebras $\1$ and $\2$. 
We refer to \cite{KTT}. 

Recall that a Lie algebra of dimension $k>2$ is called the 
\textit{Lie algebra of the real hyperbolic space} $\RH^k$, 
denoted by $\g_{\RH^k}$, if it has a basis $\{ e_1, \ldots, e_k \}$ 
whose bracket relations are given by 
\begin{align}
[e_1, e_i] = e_i \quad (i =2, \ldots, k) . 
\end{align}
We consider the direct sums of such Lie algebras and abelian Lie algebras. 
For the later convenience, we use the following bases: 
\begin{align}\label{1}
\begin{split}
\1 & 
= \mathrm{span}_{\mathbb{R}} \{ e_1, \ldots , e_n \mid [e_1, e_2] = e_2 \} , \\ 
\2 & 
= \mathrm{span}_{\mathbb{R}} \{ e_1, \ldots , e_n \mid [e_1, e_i] = e_i \ 
(i=3, \ldots, n) \} . 
\end{split}
\end{align}

In order to express $\PM$, one needs $\A$. 
Hence, let us study its Lie algebra $\D$. 
Recall that 
\begin{align}
\mathrm{Der}(\g) &:= \{ D \in \mathfrak{gl}(\g) \mid 
D[\cdot, \cdot] = [D(\cdot), \cdot] + [\cdot, D(\cdot)] \}, \\ 
\mathbb{R} & := \{ c \cdot \mathrm{id} : \g \rightarrow \g \mid c \in \mathbb{R} \}.
\end{align}

Note that the above two Lie algebras have been studied in \cite{KTT}. 
In fact, the derivation algebras $\mathrm{Der}(\g)$ have been described 
in the proof of \cite[Proposition~4.6]{KTT}. 
It then shows, under a suitable change of basis, the following 
(one can also check it by direct calculations). 
Recall that we identify $\g \cong \R^n$. 
We say that a linear map $\varphi : \g \to \g$ has a matrix expression $A$ 
with respect to a basis $\{ x_1, \ldots, x_n \}$ if it satisfies 
\begin{align} 
(\varphi(x_1) , \ldots, \varphi(x_n)) = (x_1, \ldots, x_n) A . 
\end{align} 

\begin{lem}[{cf.\ \cite[Proposition~4.6]{KTT}}]
\label{Der}
Let $\g = \1$ or $\2$. 
Then, with respect to the bases in (\ref{1}), 
we have the following matrix expressions$:$
\begin{align}
\displaystyle \Der = \left\{ \left(
\begin{array}{cc|ccc}
0 & 0 & 0 & \cdots & 0 \\
* & * & 0 & \cdots & 0 \\ \hline
* & 0 & & & \\
\vdots & \vdots & & * & \\
* & 0 & & &
\end{array}
\right) \right\} . 
\end{align}
\end{lem}

It is remarkable that these Lie algebras have the same $\D$. 
This is a reason for the choice of the bases in (\ref{1}). 

\subsection{A set of representatives } 

In this subsection, 
we give a set of representatives of $\PM$ for $\g = \1$ and $\2$. 
Let $(\A)^0$ be the connected component of $\A$ containing the identity. 
By Lemma \ref{Der}, one knows $\D$ for our Lie algebras. 
Hence, by exponentiating it, we obtain 
\begin{align}\label{Aut}
\displaystyle (\A)^0 
= \left\{ \left(
\begin{array}{cc|ccc}
x_1 & 0 & 0 & \cdots & 0 \\
\ast & x_2 & 0 & \cdots & 0 \\ \hline
\ast & 0 & & & \\
\vdots & \vdots & & B & \\
\ast & 0 & & &
\end{array}
\right) \mid x_1, x_2>0, \ \det B > 0 \right\} . 
\end{align}

The next proposition gives a set of representatives of $\PM$. 
Let $\naiseki_0$ be the inner products 
such that the bases in (\ref{1}) are orthonormal. 
Denote by $I_n$ the identity matrix, and by $E_{i,j}$ the matrix 
whose $(i,j)$-entry is $1$ and others are $0$. 

\begin{prop}\label{pm}
Let $\g = \1$ or $\2$. 
Then the following $U$ is a set of representatives of $\PM${$:$}
\begin{align}\label{ExPM}
U := \{ g_{\lambda} := I_n - \lambda E_{n,2} \mid \lambda \geq 0 \} . 
\end{align}
\end{prop}

\begin{proof}
We prove this simultaneously for $\g = \1$ and  $\2$, 
since they have the same $\D$. 
Take any $g \in \GL$. 
By Lemma~\ref{lem:representatives}, 
we have only to prove that 
\begin{eqnarray}
\label{eq:pm-expression}
\exists 
g_{\lambda} \in U : g_{\lambda} \in [[ g ]] . 
\end{eqnarray} 
First of all, 
from linear algebra, there exists $\varphi_1 \in \mathrm{O}(n)$ such that 
$g \varphi_1$ is lower triangular, 
and all diagonal entries are positive. 
We denote this by 
\begin{align} 
\left( 
\begin{array}{c|c} 
A_1 & 0 \\ \hline 
A_3 & A_4 
\end{array} 
\right) 
:= g \varphi_1 \in [[g]] , 
\end{align} 
where $A_1 \in \mathrm{GL}_2(\R)$ and $A_4 \in \mathrm{GL}_{n-2}(\R)$. 
Note that $A_1$ and $A_4$ are lower triangular, 
and all diagonal entries are positive. 
Then, it follows from (\ref{Aut}) that 
\begin{align}
\varphi _2 := \left(
\begin{array}{c|c}
A_1^{-1} & 0 \\ \hline
0 & A_4^{-1}
\end{array}
\right) \in \A . 
\end{align}
This gives  
\begin{align}
[[g]] \ni \varphi _2 g \varphi _1 = \left(
\begin{array}{c|c}
I_2 & 0 \\ \hline
A_4^{-1} A_3 & I_{n-2}
\end{array}
\right) =: g^{(1)} . 
\end{align}
We put $(v_1, v_2) := A_4^{-1} A_3$, where $v_1, v_2 \in \R^{n-2}$. 
Then again (\ref{Aut}) yields that 
\begin{align}
[[g]] \ni 
\left(
\begin{array}{cc|c}
1 & 0 & 0 \\
0 & 1 & 0 \\ \hline
-v_1 & 0 & I_{n-2} 
\end{array} 
\right) 
g^{(1)} = \left(
\begin{array}{cc|c}
1 & 0 & 0 \\
0 & 1 & 0 \\ \hline
0 & v_2 & I_{n-2} 
\end{array} 
\right) =: g^{(2)} . 
\end{align}
Here one know that there exist $B \in \mathrm{SO}(n-2)$ and $\lambda \geq 0$ 
such that 
\begin{align} 
B v_2
= {}^t ( 0 , \ldots , 0 , - \lambda ) . 
\end{align}
Note that 
\begin{align}
\varphi_3 
:= \left(
\begin{array}{c|c}
I_2 & 0 \\ \hline
0 & B
\end{array}
\right) \in (\A) \cap \mathrm{O}(n) . 
\end{align} 
This concludes that 
\begin{align}
[[g]] \ni \varphi_3 g^{(2)} \varphi_3^{-1} = g_{\lambda} , 
\end{align} 
which completes the proof of the proposition. 
\end{proof}

\subsection{Examples of Milnor-type theorems} 

In this subsection, we obtain Milnor-type theorems for our two Lie algebras. 
Here we need to study them individually. 
We start with the case $\g = \1$. 

\begin{prop}\label{EMTT}
Let $\g=\1$. 
Then, for every inner product $\naiseki$ on $\g$, 
there exist $\lambda \geq 0$, $k > 0$, 
and an orthonormal basis $\{x_1, \ldots , x_n\}$ with respect to $k \naiseki$ 
such that the bracket relations are given by 
\begin{align}\label{bp}
[x_1, x_2] = x_2 + \lambda x_n . 
\end{align}
\end{prop}

\begin{proof}
Let $\{ e_1 , \ldots , e_n \}$ be the canonical basis of $\g$ defined in (\ref{1}), 
and $\naiseki_0$ be the inner product so that this basis is orthonormal. 
Recall that, by Proposition~\ref{pm}, the set 
\begin{align}
\label{eq:U}
U = \{ g_{\lambda} = I_n - \lambda E_{n,2} \mid \lambda \geq 0 \} 
\end{align}
is a set of representatives of $\PM$. 
Take any inner product $\naiseki$ on $\g$. 
By Theorem~\ref{MTT}, 
there exist $g_{\lambda} \in U$, $k>0$, and $\varphi \in \Aut$ such that 
$\{ \varphi g_{\lambda} e_1 , \ldots , \varphi g_{\lambda} e_n \}$ 
is orthonormal with respect to $k \naiseki$. 
We put $x_i := \varphi g_{\lambda} e_i$ for 
$i \in \{ 1, \ldots , n \}$. 
It is clear that the basis $\{ x_1, \ldots, x_n \}$ is orthonormal 
with respect to $k \naiseki$. 
Hence, we have only to check the bracket relations among them. 
Note that 
\begin{align}\label{ge}
g_\lambda e_i 
= \left\{
\begin{array}{@{\,}ll}
e_i & (i \neq 2) , \\
e_2 - \lambda e_n & (i = 2) . 
\end{array}
\right. 
\end{align}
Recall that $[e_1, e_2] = e_2$ is the only nonzero bracket relation
with respect to $\{ e_1, \ldots, e_n \}$. 
We thus obtain 
\begin{align} 
[g_\lambda e_1 , g_\lambda e_2] = [e_1 , e_2 - \lambda e_n] = e_2 
= g_\lambda e_2 + \lambda g_\lambda e_n . 
\end{align} 
Since $\varphi \in \Aut$, we obtain 
\begin{align} 
[x_1, x_2] 
= [\varphi g_\lambda e_1 , \varphi g_\lambda e_2] 
= \varphi [g_\lambda e_1 , g_\lambda e_2] 
= x_2 + \lambda x_n . 
\end{align} 
It remains to show that this is the only nonzero bracket relation. 
Take any $i<j$, and assume that $j \geq 3$. 
Then one has 
\begin{align} 
[g_\lambda e_i , g_\lambda e_j] = [g_\lambda e_i , e_j] = 0 . 
\end{align} 
This yields $[x_i, x_j] = 0$, which completes the proof. 
\end{proof}

We next study the case $\g = \2$. 
The argument is the same as the former case.  

\begin{prop}\label{EMTT2}
Let $\g=\2$. 
Then, for every inner product $\naiseki$ on $\g$, there exist 
$\lambda \geq 0$, 
$k > 0$, 
and an orthonormal basis $\{ x_1, \ldots, x_n \}$ with respect to $k \naiseki$ 
such that the bracket relations are given by 
\begin{align}
[x_1, x_2] = -\lambda x_n , \quad [x_1, x_i] = x_i \quad 
(\mbox{for $i \in \{ 3, \ldots , n \}$}) . 
\end{align}
\end{prop}

\begin{proof}
Let $\{ e_1 , \ldots , e_n \}$ be the canonical basis of $\g$ defined in (\ref{1}). 
Let $g_{\lambda} \in U$. 
Note that (\ref{ge}) also holds for this case. 
Then, one has 
\begin{align} 
\begin{split} 
[g_\lambda e_1 , g_\lambda e_2] 
& = [e_1 , e_2 - \lambda e_n] 
= - \lambda e_n 
= - \lambda g_\lambda e_n , \\ 
[g_\lambda e_1 , g_\lambda e_i] 
& = [e_1 , e_i] 
= e_i 
= g_\lambda e_i 
\end{split} 
\end{align} 
for $i \in \{ 3, \ldots , n \}$, 
and others are equal to zero. 
By applying $\varphi \in \Aut$, one can complete the proof. 
\end{proof}

\begin{rem} 
For the case of $n=3$, 
the Lie algebra $\g_{\RH^2} \oplus \R$ has been studied in \cite{HL}. 
In the proofs of \cite[Lemma~5.3, Theorem~5.6]{HL}, they showed the following: 
for every inner product $\naiseki$ on $\g_{\RH^2} \oplus \R$, 
there exist $\mu, \nu >0$ and an orthonormal basis $\{ Y_1, Y_2, Y_3 \}$ 
whose bracket relations are given by 
\begin{align} 
[Y_1, Y_2] = (\sqrt{\mu} / \sqrt{\nu}) Y_3 , \quad 
[Y_1, Y_3] = (2 / \sqrt{\nu}) Y_3 , 
\end{align} 
or, 
by putting $\kappa := 1 / (2 \sqrt{\nu})$, 
\begin{align} 
[Y_1, Y_2] = \kappa Y_2 - \sqrt{3} \, \kappa Y_3 , \quad 
[Y_1, Y_3] = - \sqrt{3} \, \kappa Y_2 + 3 \kappa Y_3 . 
\end{align} 
By comparing with these relations, 
we could say that 
the case of $n=3$ of our Milnor-type theorem has a simpler expression. 
The Lie algebra $\g_{\RH^2} \oplus \R$ has been also studied in \cite{Ch}. 
For more details, we refer to \cite[Lemma~5]{Ch}.
\end{rem} 

\begin{rem} 
For the case of $n=4$, 
the geometry of left-invariant Riemannian metrics 
have been studied in \cite{KN1, KN2}. 
For the Lie algebra $\g_{\RH^2} \oplus \R^2$, in \cite[Lemma~6]{KN2}, they showed the following: 
for every inner product $\naiseki$ on $\g_{\RH^2} \oplus \R^2$, 
there exist $a>0$, $b \geq 0$, 
and an orthonormal basis $\{ f_1, f_2, f_3 \}$ 
whose bracket relations are given by 
\begin{align} 
[f_1, f_2] = a f_2 + b f_3 . 
\end{align} 
For the Lie algebra $\g_{\RH^3} \oplus \R$, 
in \cite[Lemma~9]{KN2}, they also showed the following: 
for every inner product $\naiseki$ on $\g_{\RH^3} \oplus \R$, 
there exist $a>0$, $b \geq 0$, 
and an orthonormal basis $\{ f_1, f_2, f_3 \}$ 
whose bracket relations are given by 
\begin{align} 
[f_1, f_3] = a f_1 , \quad [f_2, f_3] = a f_2 + b f_4 . 
\end{align} 
These results are quite similar to ours. 
In fact, our results imply that, up to scaling, 
we can assume $a=1$ in both cases.  
\end{rem} 

We recall that Theorem~\ref{MTT} gives a procedure 
to obtain a Milnor-type theorem for any Lie algebra. 
As mentioned above, it recovers 
(and sometimes simplifies)  
some known results in \cite{Ch, HL, KN1, KN2}. 
Moreover, we emphasize that it 
will provide a new Milnor-type theorem, 
which would be useful to study the geometry of left-invariant Riemannian metrics. 

\section{Applications} 
\label{sec.4}

A Milnor-type theorem can be applied to study the geometry of left-invariant metrics. 
In this section, we see examples of such studies. 
Namely, for our two Lie algebras, we determine the possible Ricci signatures, and classify solvsolitons on them. 

Throughout this section, we identify a metric Lie algebra $(\g, \naiseki)$ 
with the simply-connected Lie group equipped with the induced left-invariant 
Riemannian metric. 

\subsection{Calculations of the curvatures} 

In this subsection, 
we calculate the Ricci operators of $\1$ and $\2$ with respect to arbitrary inner products. 

First of all, 
we recall the curvatures of a metric Lie algebra $(\g, \naiseki)$. 
Let $X, Y, Z \in \g$. 
Then the Levi-Civita connection $\nabla$ is given by 
\begin{align} 
2 \langle \nabla _X Y , Z \rangle 
= \langle [Z,X] , Y \rangle + \langle X , [Z,Y] \rangle +\langle [X,Y] , Z \rangle . 
\end{align} 
The Riemannian curvature $R$ is defined by 
\begin{align} 
R(X,Y)Z := \nabla _X \nabla _Y Z - \nabla _Y \nabla _X Z - \nabla _{[X,Y]} Z.
\end{align}
Let $\{ e_1, \ldots, e_n \}$ be an orthonormal basis of $\g$ with respect to $\naiseki$. 
The Ricci operator 
$\mathrm{Ric}_{\naiseki} : \g \rightarrow \g$ is defined by
\begin{align}
\mathrm{Ric}_{\naiseki}(X):=\sum R(X,e_i)e_i.
\end{align}

We now consider the case of $\g = \1$. 
For any inner product $\naiseki$ on it, 
we calculate the curvatures in terms of the basis given in Proposition~\ref{EMTT}. 

\begin{prop}\label{ric1}
Let $\naiseki$ be an inner product on $\1$, 
and assume that there exist $\lambda \geq 0$ and 
an orthonormal basis $\{ x_1, \ldots, x_n \}$ with respect to $\naiseki$ 
such that the bracket relations are given by 
\begin{align} 
[x_1, x_2] = x_2 + \lambda x_n . 
\end{align} 
Then the Ricci operator satisfies 
\begin{align}
\mathrm{Ric}_{\naiseki}(x_i) 
= \left\{ 
\begin{array}{@{\,}ll} 
- (1 + (\lambda^2/2) ) \, x_i & (i = 1 , 2) , \\ 
0 & (i = 3 , \ldots , n-1) , \\ 
(\lambda^2/2) \, x_n & (i = n) . 
\end{array} 
\right. 
\end{align}
\end{prop}

\begin{proof}
First of all, we calculate the Levi-Civita connection $\nabla$. 
By the bracket relations, one can see that 
$x_3, \ldots, x_{n-1}$ do not give any effects on $\nabla$. 
Namely, 
$\langle \nabla_{x_i} x_j , x_k \rangle \neq 0$ 
implies 
$i, j, k \in \{ 1, 2, n \}$. 
A direct calculation shows that nonzero components of $\nabla$ are precisely 
\begin{align*} 
\nabla_{x_1} x_2 &= (\lambda / 2) x_n , & 
\nabla_{x_2} x_1 &= - x_2 - (\lambda / 2) x_n , \\ 
\nabla_{x_1} x_n &= - (\lambda / 2) x_2 , & 
\nabla_{x_n} x_1 &= - (\lambda / 2) x_2 , \\
\nabla_{x_2} x_2 &= x_1 , \quad & & \\
\nabla_{x_2} x_n &= (\lambda / 2) x_1 , \quad & 
\nabla_{x_n} x_2 &= (\lambda / 2) x_1 . 
\end{align*} 
Note that, by using the torsion-free condition 
$\nabla_x y - \nabla_y x = [x,y]$, 
the right-hand columns can be obtained from the left-hand ones. 

One can thus calculate the Riemannian curvatures $R$. 
The nonzero components of $R$, 
which we need for calculating the Ricci operator, are 
\begin{align*} 
R(x_1, x_2) x_2 &= - (1 + (3/4) \lambda^2) x_1, \ &
R(x_1, x_n) x_n &= (1/4) \lambda^2 x_1, \\
R(x_2, x_1) x_1 &= - (1 + (3/4) \lambda^2) x_2, \ &
R(x_2, x_n) x_n &= (1/4) \lambda^2 x_2, \\
R(x_n, x_1) x_1 &= (1/4) \lambda^2 x_n, \ &
R(x_n, x_2) x_2 &= (1/4) \lambda^2 x_n. 
\end{align*} 
By summing up these components, one can calculate the Ricci operator, 
which completes the proof of the proposition. 
\end{proof}

Next we consider the case of $\g=\2$. 
The calculation is similar to the former case. 

\begin{prop}\label{ric2}
Let $\naiseki$ be an inner product on $\2$, 
and assume that there exist 
$\lambda \geq 0$ and an orthonormal basis $\{ x_1, \ldots, x_n \}$ with respect to $\naiseki$ 
such that the bracket relations are given by 
\begin{align}
[x_1, x_2] = - \lambda x_n , \quad [x_1, x_i] = x_i \quad 
(\mbox{for $i \in \{ 3, \ldots , n \}$}) . 
\end{align}
Then the Ricci operator satisfies 
\begin{align} 
\mathrm{Ric}_{\naiseki}(x_i) 
= \left\{ 
\begin{array}{@{\,}ll} 
- (n-2 + (\lambda^2/2) ) \, x_1 & (i = 1) , \\
- (\lambda^2/2) \, x_2 + ((n-1)\lambda /2) \, x_n & (i = 2) , \\
-(n-2) \, x_i & (i = 3 , \ldots , n-1) , \\ 
((n-1)\lambda /2) \, x_2 + ((\lambda^2/2) -(n-2)) \, x_n & (i = n) . 
\end{array} 
\right. 
\end{align}
\end{prop}

\begin{proof}
First of all, we calculate nonzero components of $\nabla$.  
Let $i \in \{ 3, \ldots, n-1 \}$. 
By a similar calculation as before, we have 
\begin{align*}
\nabla_{x_1} x_2 &= - (\lambda / 2) x_n , & 
\nabla_{x_2} x_1 &= (\lambda / 2) x_n , \\ 
& & \nabla_{x_i} x_1 &= - x_i , \\ 
\nabla_{x_1} x_n &= (\lambda / 2) x_2 , & 
\nabla_{x_n} x_1 &= (\lambda / 2) x_2 - x_n , \\ 
\nabla_{x_2} x_n &= - (\lambda / 2) x_1 , & 
\nabla_{x_n} x_2 &= - (\lambda / 2) x_1 , \\ 
\nabla_{x_i} x_i &= x_1 , & 
\nabla_{x_n} x_n &= x_1 . 
\end{align*} 
Then one has
\begin{align*} 
R(x_1, x_2)\, x_2 &= - (3/4) \lambda^2 x_1 , \\ 
R(x_1, x_i)\, x_i &= - x_1 \qquad (\mbox{for $i = 3, \ldots, n-1$}) , \\ 
R(x_1, x_n)\, x_n &= - (1 - (1/4) \lambda^2) x_1 . 
\end{align*} 
By summing up them, 
one can show the assertion on $\mathrm{Ric}_{\naiseki}(x_1)$. 
Similarly, direct calculations yield that 
\begin{align*} 
R(x_2, x_1) x_1 & = - (3/4) \lambda^2 x_2 + \lambda x_n, \\
R(x_2, x_i) x_i & = (\lambda /2) x_n \qquad (\mbox{for $i=3, \ldots, n-1$}) , \\ 
R(x_2, x_n) x_n & = (1/4) \lambda^2 x_2 . 
\end{align*} 
This proves the assertion on $\mathrm{Ric}_{\naiseki}(x_2)$. 
The remaining assertions follow from
\begin{align*} 
R(x_i, x_j) x_j &= - x_i \qquad 
(\mbox{for $i=3, \ldots, n-1$ and $j \neq 2, i$}) , \\ 
R(x_n, x_1) x_1 &= \lambda x_2 - (1 - (1/4) \lambda^2) x_n, \\
R(x_n, x_2) x_2 &= (1/4) \lambda^2 x_n , \\ 
R(x_n, x_i) x_i &= (\lambda /2) x_2 - x_n \qquad (\mbox{for $i=3, \ldots, n-1$}) . 
\end{align*} 
This completes the proof of the proposition. 
\end{proof}

\subsection{Ricci signatures} 

In this subsection, 
we determine all possible 
signatures for the Ricci curvatures of $\1$ and $\2$. 
For the notational conventions, we say that 
a metric Lie algebra $(\g , \naiseki)$ has the 
\textit{Ricci signature} 
$(-, 0, +)=(m_1, m_2, m_3)$ 
if the numbers of negative, zero and positive eigenvalues of 
the Ricci operator $\mathrm{Ric}_{\naiseki}$ 
are equal to $m_1$, $m_2$ and $m_3$, respectively. 

\begin{prop}\label{sig}
We have the following$:$ 
\begin{enumerate}
\item 
For $\g=\1$, the possible Ricci signatures are 
$(-, 0, +)=(2, n-2, 0)$ and $(2, n-3, 1)$. 
\item 
For $\g=\2$, the possible Ricci signatures are 
$(-, 0, +)=(n-1, 1, 0)$ and $(n-1, 0, 1)$.
\end{enumerate}
\end{prop}

\begin{proof}
We show (1). 
Take any inner product $\naiseki$ on $\1$. 
Recall that, by the Milnor-type theorem (Proposition \ref{EMTT}), 
there exist $\lambda \geq 0$, 
$k > 0$ and an orthonormal basis $\{ x_1, \ldots, x_n \}$ 
with respect to $k \naiseki$ 
such that the bracket relations are given by 
\begin{align}
[x_1, x_2] = x_2 + \lambda x_n . 
\end{align} 
We can assume $k=1$ because the Ricci signature is invariant under scaling. 
By Proposition~\ref{ric1}, 
it is easy to see that 
\begin{align} 
(-, 0, +)=\left\{
\begin{array}{@{\,}ll}
(2, n-2, 0) & \mbox{if} \  \lambda = 0 , \\
(2, n-3, 1) & \mbox{if} \  \lambda \neq 0 . 
\end{array}
\right.
\end{align}
This proves (1). 
In order to show (2), take any $\naiseki$ on $\2$. 
Without loss of generality, 
we can take an orthonormal basis as described in Proposition~\ref{ric2}. 
Then $x_1, x_3, \ldots, x_{n-1}$ are eigenvectors with negative eigenvalues. 
It remains to see the $\mathrm{span} \{ x_2, x_n \}$-direction. 
For this direction, it suffices to calculate the eigenvalues of 
\begin{align} 
\displaystyle A := 
\left( 
\begin{array}{cc} 
- \lambda^2/2 & ((n-1)\lambda /2) \\ 
((n-1)\lambda /2) & ((\lambda^2 /2)-(n-2)) 
\end{array}
\right) . 
\end{align} 
Then, the eigen-polynomial of $2A$ is 
\begin{align} 
\det(t I_2 - 2A) = t^2 + 2(n-2) t - (n-1)^2 \lambda^2 . 
\end{align} 
One thus can see that the eigenvalues of $2A$ are 
$(-,+)$ if $\lambda >0$, and $(-,0)$ if $\lambda =0$. 
This completes the proof.
\end{proof}

As a corollary of Proposition~\ref{sig}, 
one can immediately see that $\1$ and $\2$ with $n \geq 3$ 
do not admit left-invariant Einstein metrics. 

Possible Ricci signatures in the cases of $n=3,4$ 
have been known by \cite{Ch, HL, KN2}. 

\subsection{Solvsolitons}

In this subsection, we classify solvsolitons on the Lie algebras $\1$ and $\2$. 
In fact, they admit solvsolitons for any $n \geq 3$. 

First of all, we recall the notion of solvsolitons following Lauret \cite{Lau11}. 

\begin{defi}
An inner product $\naiseki$ on a solvable Lie algebra $\g$ is called a 
\textit{solvsoliton} 
if there exist $c \in \mathbb{R}$ and $D \in \mathrm{Der}(\g)$ such that 
\begin{align}
\mathrm{Ric}_{\naiseki} = c I + D , 
\end{align}
where 
$I$ is the identity map of $\g$. 
\end{defi}

It is remarkable that solvsolitons give rise to left-invariant Ricci solitons. 
For deeper discussions and further results on solvsolitons, 
we refer the reader to \cite{Lau11} and references therein. 

Our Milnor-type theorems are useful to classify solvsolitons on a given Lie algebra. 
We demonstrate it by our two Lie algebras, $\g = \1$ and $\2$. 
Recall that $\naiseki_0$ is the inner product 
such that the canonical basis $\{ e_1, \ldots, e_n\}$ is orthonormal. 

\begin{prop}\label{sol}
Let $\g=\1$ or $\2$. 
Then, an inner product $\naiseki$ on $\g$ is a solvsoliton 
if and only if $[\naiseki] = [\naiseki_0]$. 
\end{prop}

\begin{proof} 
We first prove the case $\g=\1$. 
Take any inner product $\naiseki$ on $\g$. 
Then, by Proposition~\ref{EMTT}, 
there exist $\lambda \geq 0$, $k > 0$, 
and an orthonormal basis $\{ x_1, \ldots, x_n \}$ 
with respect to $k \naiseki$ 
such that the bracket relations are given by 
\begin{align} 
[x_1, x_2] = x_2+ \lambda x_n . 
\end{align} 

Our first claim is that the matrix expression of $\Der$ 
with respect to the basis $\{x_1, \ldots, x_n \}$ is 
\begin{align}\label{Derivation} 
\displaystyle \Der = \left\{ \left( 
\begin{array}{cc|cccc} 
0 & 0 & 0 & \cdots & \cdots & 0 \\ 
x_{21} & x_{22} & 0 & \cdots & \cdots & 0 \\ \hline 
x_{31} & -\lambda x_{3 n} & x_{33} & \cdots & \cdots & x_{3 n} \\ 
\vdots & \vdots & \vdots & \ddots &  & \vdots \\ 
\vdots & -\lambda x_{n-1, n} & \vdots &  & \ddots & \vdots \\ 
x_{n 1} & \lambda (x_{22}-x_{n n}) & x_{n 3} & \cdots & \cdots & x_{n n} 
\end{array} 
\right) \right\} . 
\end{align}
We use Lemma~\ref{Der}, 
which gives the matrix expression of $\Der$ 
with respect to $\{ e_1, \ldots, e_n \}$. 
In order to describe the change of basis matrix, we recall that 
$g_\lambda := I_n - \lambda E_{n,2}$. 
Let us define 
\begin{align}\label{change of basis} 
(x^\prime_1 , \ldots , x^\prime_n) := (e_1, \ldots, e_n) g_\lambda . 
\end{align} 
Then, the bases $\{ x^\prime_1 , \ldots , x^\prime_n \}$ 
and $\{ x_1 , \ldots , x_n \}$ have the same bracket relations. 
Hence, it is sufficient to calculate the matrix expression of $\Der$ 
with respect to 
$\{ x^\prime_1 , \ldots , x^\prime_n \}$. 
Let $f : \g \to \g$ be a linear map, and 
denote by $D$ the matrix expression of $f$ with respect to $\{ e_1 , \ldots , e_n \}$. 
Then, we have
\begin{align}
( f(x_1^{\prime}), \ldots, f(x_n^{\prime}) ) 
&= ( f(e_1), \ldots, f(e_n) ) g_{\lambda} 
= (x_1^{\prime}, \ldots, x_n^{\prime}) g_{\lambda}^{-1} D g_{\lambda} . 
\end{align} 
That is, $g_{\lambda}^{-1} D g_{\lambda}$ 
is the matrix expression of $f$ with respect to 
$\{ x_1^\prime , \ldots, x_n^\prime \}$. 
By using this, one can complete the proof of our claim. 

We show the ``only if''-part. 
Assume that $\naiseki$ is a solvsoliton. 
We can assume $k=1$ because the solvsoliton is preserved by scaling. 
Hence, by Proposition~\ref{ric1}, 
the matrix expression of $\mathrm{Ric}_{\naiseki}$ is 
\begin{align} 
\label{eq:ric1} 
\left( 
\begin{array}{cc|cccc} 
-1-(\lambda^2/2) & & & & & \\ 
& -1-(\lambda^2/2) & & & & \\ \hline 
& & 0 & & & \\ 
& & & \ddots & & \\ 
& & & & 0 & \\ 
& & & & & \lambda^2/2 
\end{array} 
\right) . 
\end{align}
By assumption on $\naiseki$, 
there exist $c \in \mathbb{R}$ and $D \in \Der$ such that 
\begin{align} 
\mathrm{Ric}_{\naiseki} = c I + D . 
\end{align} 
Then, by comparing (\ref{eq:ric1}) and (\ref{Derivation}), 
we obtain 
$c = -1-(\lambda^2/2)$, and $\lambda = 0$. 
This yields that two bases $\{ x_1 , \ldots , x_n \}$ and $\{ e_1 , \ldots , e_n \}$ 
have the same bracket relations. 
Recall that these bases are orthonormal with respect to 
$\naiseki$ and $\naiseki_0$, respectively. 
This concludes $[\naiseki]=[\naiseki_0]$. 

We show the ``if''-part. 
Assume that $[\naiseki]=[\naiseki_0]$. 
Then, one can take the basis $\{ x_1, \ldots, x_n \}$ so that $\lambda=0$. 
Hence, by substituting $\lambda=0$ into (\ref{eq:ric1}) and (\ref{Derivation}), 
one can easily see that $\naiseki$ is a solvsoliton with $c=-1$. 

The proof of the case $\g=\2$ is similar to the former case. 
Let $\{x_1, \ldots, x_n \}$ be a basis whose bracket relations are given in Proposition~\ref{EMTT2}. 
Then, by the same argument as above, 
the matrix expression of $\Der$ with respect to this basis 
coincides with (\ref{Derivation}). 
Hence, by using Proposition~\ref{ric2}, 
one can classify solvsolitons on $\g=\2$. 
\end{proof}

\end{document}